\documentclass[11pt,twoside]{article}
\usepackage{amsmath,amsfonts,amssymb,theorem,amscd}
\usepackage{verbatim}

\numberwithin{equation}{section}

\newcommand{\calC}{{\mathcal C}}

\newtheorem{theorem}{Theorem}[section]
\newtheorem{lemma}{Lemma}[section]
\newtheorem{proposition}{Proposition}[section]
\newtheorem{corollary}{Corollary}[section]
\newtheorem{definition}{Definition}[section]

\newtheorem{remark}{Remark}[section]
\usepackage[usenames,dvipsnames]{pstricks}
\usepackage{epsfig}
\usepackage{pst-grad} 
\usepackage{pst-plot} 

\def\Xint#1{\mathchoice
   {\XXint\displaystyle\textstyle{#1}}%
   {\XXint\textstyle\scriptstyle{#1}}%
   {\XXint\scriptstyle\scriptscriptstyle{#1}}%
   {\XXint\scriptscriptstyle\scriptscriptstyle{#1}}%
   \!\int}
\def\XXint#1#2#3{{\setbox0=\hbox{$#1{#2#3}{\int}$}
     \vcenter{\hbox{$#2#3$}}\kern-.5\wd0}}

\def\aint{\Xint\diagup}

\newcommand{\qed}{\hfill\fbox{}\par\vspace{.2cm}}
\newcommand{\bke}[1]{\left( #1 \right)}
\newcommand{\bkt}[1]{\left[ #1 \right]}
\newcommand{\bket}[1]{\left\{ #1 \right\}}
\newcommand{\norm}[1]{\left\Vert #1 \right\Vert}
\newcommand{\abs}[1]{\left| #1 \right|}

\newenvironment{proof}{{\bf Proof.}} {\hfill\qed}

\newenvironment{pfthm1}{{\par\noindent\bf
           Proof of Theorem \ref{main-thm-boundary-v}. }}{\hfill\fbox{}\par\vspace{.2cm}}

\newenvironment{pfthm2}{{\par\noindent\bf
           Proof of Theorem \ref{main-thm-boundary-w}. }}{\hfill\fbox{}\par\vspace{.2cm}}

\newenvironment{pfthm3}{{\par\noindent\bf
           Proof of Theorem \ref{main-theorem-interior}. }}{\hfill\fbox{}\par\vspace{.2cm}}

\newenvironment{pf-mhd-decay}{{\par\noindent\bf
           Proof of Lemma \ref{mhd-decay}. }}{\hfill\fbox{}\par\vspace{.2cm}}

\newcommand{\R}{{ \mathbb{R}  }}

\begin{document}
\renewcommand{\thesection}{\arabic{section}}

\newcommand{\pr}{\partial}
\newcommand{\nl}{\vskip 1pc}
\newcommand{\co}{\mbox{co}}

\newcommand{\calT} {\mathcal T}

\title{Boundary regularity criteria for suitable weak solutions of
the magnetohydrodynamic equations}

\author{Kyungkeun Kang and Jae-Myoung Kim}
\date{}
\maketitle

\begin{abstract}
We present some new regularity criteria for suitable weak solutions
of magnetohydrodynamic equations near boundary in dimension three.
We prove that suitable weak solutions are H\"older continuous near
boundary provided that either the scaled $L^{p,q}_{x,t}$-norm of the
velocity with $3/p+2/q\le 2$, $2<q<\infty$, or the scaled
$L^{p,q}_{x,t}$-norm of the vorticity with $3/p+2/q\le 3$,
$2<q<\infty$ are sufficiently small near the boundary.
\end{abstract}

\section{Introduction}

We study the regularity problem for {\it suitable }weak solutions
$(u, b, \pi): Q_T\rightarrow \R^3\times\R^3\times\R$ of the
three-dimensional incompressible magnetohydrodynamic (MHD) equations
\begin{equation}\label{MHD}
\left\{
\begin{array}{ll}
\displaystyle u_t- \triangle u +(u \cdot
\nabla) u - (b\cdot \nabla) b  +\nabla \pi= 0\\
\vspace{-3mm}\\
\displaystyle b_t - \triangle b +(u \cdot
\nabla) b - (b\cdot \nabla) u  = 0\\
\vspace{-3mm}\\
\displaystyle \text{div} \ u =0 \quad \text{and}\quad \text{div} \
b=0 ,\\
\vspace{-3mm}\\
\displaystyle u(x,0)=u_0(x), \quad b(x,0)=b_0(x)
\end{array}\right.
\,\,\, \mbox{ in } \,\,Q_T:=\R^3_{+}\times [0,\, T).
\end{equation}
Here $u$ is the flow velocity vector, $b$ is the magnetic vector and
$\displaystyle\pi=p+ \frac{\abs{b}^2}{2}$ is the magnetic pressure.
The boundary conditions of $u$ and $b$ are given as no-slip and slip
conditions, respectively, namely
\begin{equation}\label{noslip-slip}
u=0 \quad \text{and}\quad b\cdot \nu=0,\, \ (\nabla \times b)\times
\nu=0,\qquad \mbox{ on }\,\,\partial\R^3_+,
\end{equation}
where $\nu=(0,0,-1)$ is the outward unit normal vector along
boundary $\partial\R^3_+$. By suitable weak solutions we mean
solutions that solve MHD equations in the sense of distribution and
satisfy the local energy inequality (see Definition \ref{sws-3dnse}
in section 2 for details).

The MHD equations describe the dynamics of the interaction of moving
conducting fluids with electro-magnetic fields which are frequently
observed in nature and industry, e.g., plasma liquid metals, gases,
two-phase mixtures (see e.g. \cite{D} and \cite{DL}).

Let $x=(x_1,x_2,0)\in\partial\R^3_+$. For a point $z=(x,t)\in
\partial\R^3_+\times (0,T)$, we denote
\[
B_{x,r}=:\{y\in\R^3: \abs{y-x}<r\},\quad B^{+}_{x,r}:=\{y \in
B_{x,r}: y_3>0 \},
\]
\[
Q_{z,r}=:B_{x,r}\times (t-r^2,t),\quad Q^{+}_{z,r}:=\{(y,t) \in
Q_{z,r}: y_3>0 \},\quad r<\sqrt{t}.
\]
We say that solutions $u$ and $b$ are regular at $z\in
\overline{\R^3_+}\times (0,T)$ if $u$ and $b$ are H\"older
continuous for some $Q^+_{z,r}$, $r>0$. Otherwise, it is said that
$u$ and $b$ are singular at $z$.

We list some known results for MHD equations relevant to our
concern, in particular regarding regularity conditions in terms of
scaled invariant quantities.

It was shown in \cite{DL} that weak solutions for MHD equations
exist globally in time and in the two-dimensional case weak
solutions become regular (compare to \cite{L34} and \cite{H} for the
NSE). In the three-dimensional case, as shown in \cite{ST83}, if a
weak solution pair $(u,\, b)$ are additionally in $L^{\infty}(0,
 \, T; H^1({\mathbb{R}}^3))$, $(u,\, b)$ become regular.
Although many significant contributions have been made on the
existence, uniqueness and regularity of weak solutions to the MHD
equations, as in the NSE, regularity question, however, remains open
in dimension three.

In case that $\Omega=\R^3$, it was proved in \cite{He-Xin05-jde}
that a weak solution pair $(u, b)$ become regular if a certain type
of scaling invariant integral conditions for velocity field, often
referred as Serrin's condition, is additionally assumed (see e.g.
\cite{P59}, \cite{JS62}, \cite{L67}, \cite{FJR72}, \cite{Sohr83} for
the NSE). Recently, the authors have obtained similar results in the
case that $\Omega$ is a bounded domain or half space (see
\cite[Theorem 1]{KK12}) (refer to \cite{G86} for the NSE). The local
interior case of Serrin's condition including limiting case
$L^{3,\infty}_{x,t}$ was treated for MHD equations in \cite{MNS07}
(compare to \cite{ESS03}, \cite{AS07} for the NSE).

For a local case, various types of $\epsilon-$regularity criteria
for suitable weak solutions have been also established in terms of
scaled norms. Among others, it was shown in \cite{Vya10} that
suitable weak solutions become regular near a boundary point $z$ if
the following conditions are satisfied: There exists $\epsilon>0$
such that
\[
\limsup_{r\rightarrow 0}\frac{1}{r} \int_{Q^{+}_{z,r}}|\nabla
b(y,s)|^{2}dyds<\infty,\qquad \limsup_{r\rightarrow 0}\frac{1}{r}
\int_{Q^{+}_{z,r}}|\nabla u(y,s)|^{2}dyds<\epsilon.
\]
Other types of conditions in terms of scaled invariant norms near
boundary are also found in \cite{Vya11} (compare to \cite{VS82},
\cite{GAS}, \cite{K04}, \cite{VAS}, \cite{GKT06}, \cite{SSS06},
\cite{W10} for the NSE). We also refer to \cite{He-Xin05-jfa},
\cite{KL09} and \cite{Vya08} in the interior case for MHD equations
(compare to \cite{VS76}, \cite{CKN}, \cite{S88}, \cite{T90},
\cite{Lin98}, \cite{LS99}, \cite{GKT07} for the NSE).

Here we emphasize that for the global case, i.e. $\Omega=\R^3$,
additional conditions are imposed on only velocity field but not on
the magnetic field. For local interior and boundary cases, however,
known results require control of some scaled norms with scaled
factors of magnetic fields as well as those of the velocity fields.

The motivation of our study is to establish new regularity criteria
for MHD equations depending only on velocity fields for local cases.
To be more precise, main objective of this paper is to present new
sufficient conditions, not relying on magnetic fields, for the
regularity of suitable weak solutions to the MHD equations near
boundary as well as in the interior.

While preparing this paper, the authors have become to know that,
very recently, Wang and Zhang showed that local interior regularity
can be ensured by the control of only scaled norm of velocity
fields. More precisely, interior regularity criteria shown in
\cite{WZ12} is the following:
\begin{equation}\label{wz-int-rc}
\limsup_{r\rightarrow 0}r^{-(\frac{3}{p}+\frac{2}{q}-1)}
\norm{\norm{u}_{L^p(B_{x,r})}}_{L^q(t-r^2,t)}<\epsilon,
\end{equation}
where $ 1\leq \frac{3}{p}+\frac{2}{q}\leq 2$ with $1\leq q\leq
\infty$. We have also proved independently the same result as in
\cite{WZ12} and since we think that our proof is a different version
to that in \cite{WZ12}, its details are given in Appendix (see
Theorem \ref{main-theorem-interior}). Our main concern is, however,
to obtain new regularity conditions near boundary. Let
$x_0\in\partial\R^3_+$ be a boundary point in a half space. We
expect that our analysis would also hold in a smooth boundary as in
the case of flat boundary, but our study is restricted, in this
paper, to the case of $\R^3$, whose boundary is flat.

Now we are ready to state the first part of our main results.

\begin{theorem}\label{main-thm-boundary-v}
Let $(u,b,\pi)$ be a suitable weak solution of the MHD equations
\eqref{MHD} according to Definition \ref{sws-3dnse}. Suppose that
for every pair $p,q$ satisfying $\frac{3}{p}+\frac{2}{q}\leq 2$, $\
2< q \le\infty$ and $(p,q)\neq (\frac{3}{2}, \infty)$,
there exists $\epsilon>0$ depending only on $p,q$ such that for some
point $z=(x,t)\in\partial \R_{+}^3\times(0,T)$ $u$ is locally in
$L_{x,t}^{p,q}$ near $z$ and
\begin{equation}\label{b-con2}
\limsup_{r\rightarrow 0}r^{-(\frac{3}{p}+\frac{2}{q}-1)}
\norm{\norm{u}_{L^p(B^{+}_{x,r})}}_{L^q(t-r^2,t)}<\epsilon.
\end{equation}
Then, $u$ and $b$ are regular at $z$.
\end{theorem}

\begin{remark}
The result in Theorem \ref{main-thm-boundary-v} is also valid in the
interior. In fact, the range of $q$ in \eqref{b-con2} in the
interior is wider than that of boundary case. To be more precise,
the pair $(p,q)$ can be relaxed in the interior as follows:
\[
\frac{3}{p}+\frac{2}{q}\leq 2,\qquad  1\le q \le\infty,\quad
(p,q)\neq (\frac{3}{2}, \infty).
\]
As mentioned earlier, in \cite{WZ12} Wang and Zhang showed interior
regularity criteria depending only on the control of velocity fields
and we also obtain the same result independently. Since the method
of proof is a bit different to that in \cite{WZ12}, we present its
details in the Appendix for a variety of proof.
\end{remark}

Next corollaries are direct consequences of Theorem
\ref{main-thm-boundary-v}.

\begin{corollary}\label{Cor-KK1}
Let $(u,b,\pi)$ be a suitable weak solution of the MHD equations
\eqref{MHD} according to Definition \ref{sws-3dnse}. Suppose that
for some point $z=(x,t)\in\partial \R_{+}^3\times(0,T)$ $u$ is
locally in $L_{x,t}^{p,q}$ near $z$, where
$\frac{3}{p}+\frac{2}{q}=1$ with $3<p<\infty$. Then, $u$ and $b$ are
regular at $z$.
\end{corollary}

It is straightforward to prove Corollary \ref{Cor-KK1} by H\"older's
inequality, and thus we skip its details (compare to \cite{MNS07}
for local interior case). Next corollary is due to
Poincar\'e-Sobolev inequality and the details is again omitted.

\begin{corollary}\label{cor-boundary-v}
The same statement of Theorem \ref{main-thm-boundary-v} remains true
if the hypothesis including condition \eqref{b-con2} is replaced by
the following:\,\, Suppose that for every pair $p,q$ satisfying
$2\leq \frac{3}{p}+\frac{2}{q}\leq 3$, $2< q \le\infty$, and
$(p,q)\neq (1, \infty)$, there exists $\epsilon>0$ depending only on
$p,q$ such that for some point $z=(x,t)\in\partial
\R_{+}^3\times(0,T)$ $u$ is locally in $L_{x,t}^{p,q}$ near $z$ and
\begin{equation}\label{b-con44}
\limsup_{r\rightarrow 0}r^{-(\frac{3}{p}+\frac{2}{q}-1)}
\norm{\norm{\nabla u}_{L^p(B^{+}_{x,r})}}_{L^q(t-r^2,t)}<\epsilon.
\end{equation}
\end{corollary}

Considering scaling invariant quantities of vorticity, we can also
establish other regularity criteria for vorticity near boundary.

\begin{theorem}\label{main-thm-boundary-w}
Let $(u,b,\pi)$ be a suitable weak solution of the MHD equations
\eqref{MHD} according to Definition \ref{sws-3dnse}. Suppose that
for every pair $p,q$ satisfying $2\leq \frac{3}{p}+\frac{2}{q}\leq
3$, $2< q \le\infty$, and $(p,q)\neq (1, \infty)$, there exists
$\epsilon>0$ depending only on $p,q$ such that for some point
$z=(x,t)\in\partial \R_{+}^3\times(0,T)$ $\omega=\nabla\times u$ is
locally in $L_{x,t}^{p,q}$ near $z$ and
\begin{equation}\label{b-con3}
\limsup_{r\rightarrow 0}r^{-(\frac{3}{\tilde{p}}+\frac{2}{q}-2)}
\norm{\norm{\omega}_{L^{p}(B^{+}_{x,r})}}_{L^{q}(t-r^2,t)}<\epsilon.
\end{equation}
Then, $u$ and $b$ are regular at $z$.
\end{theorem}




This paper is organized as follows. In Section 2 we introduce some
scaling invariant functionals and the notion of suitable weak
solutions. In Section 3 we present the proofs of Theorem
\ref{main-thm-boundary-v} and Theorem \ref{main-thm-boundary-w}. In
the Appendix, the interior case will be treated and give a detailed
a proof with respect to $\epsilon$-regularity criteria for the
modified suitable weak solution for MHD equations.



\section{Preliminaries}

In this section we introduce some scaling invariant functionals and
suitable weak solutions, and recall
an estimation of the Stokes system.

We first start with some notations.
Let $\Omega$ be an open domain in $\R^3$
\ and $I$ be a finite time interval. For $1\le q\le \infty$, we
denote the usual Sobolev spaces by $W^{k,q}(\Omega) = \{ u \in
L^{q}( \Omega )\,:\, D^{ \alpha }u \in L^{q}( \Omega ), 0 \leq |
\alpha | \leq k \}$. As usual, $W^{k,q}_0(\Omega)$ is the completion
of $\calC^{\infty}_0(\Omega)$ in the $W^{k,q}(\Omega)$ norm. We also
denote by $W^{-k,q'}(\Omega)$ the dual space of $W^{k,q}_0(\Omega)$,
where $q$ and $q'$ are H\"older conjugates. We write the average of
$f$ on $E$ as $\aint_{E} f$, that is $\aint_{E}f =\int_{E}
f/\abs{E}$. For a function $f(x,t)$, we denote
$\|f\|_{L^{p,q}_{x,t}(\Omega\times I)}=\|f\|_{L^{q}_{t}(I;
L^p_x(\Omega))}=\|\|f\|_{L^p_x(\Omega)}\|_{L^q_t(I)}$. For vector
fields $u,v$ we write $(u_iv_j)_{i,j=1,2,3}$ as $u\otimes v$. We
denote by $C=C(\alpha,\beta,...)$ a constant depending on the
prescribed quantities $\alpha,\beta,...$, which may change from line
to line.

In this paper, we consider the case that $\Omega=\R^3_+$, i.e. a
half space in dimension three. For convenience, we denote the
boundary of $\R^3_+$ by $\Gamma=\R^3\cap\{x_3=0\}$. Next, we
introduce scaling invariant quantities near boundary. Let $z=(x,t)
\in \Gamma \times I$ and we set
\[
A_{u}(r):=\sup_{t-r^{2}\leq
s<t}\frac{1}{r}\int_{B^{+}_{x,r}}|u(y,s)|^{2}dy,\quad
E_{u}(r):=\frac{1}{r} \int_{Q^{+}_{z,r}}|\nabla u(y,s)|^{2}dyds,
\]
\[
A_{b}(r):=\sup_{t-r^{2}\leq
s<t}\frac{1}{r}\int_{B^{+}_{x,r}}|b(y,s)|^{2}dy,\quad
E_{b}(r):=\frac{1}{r} \int_{Q^{+}_{z,r}}|\nabla b(y,s)|^{2}dyds,
\]
\[
M_{u}(r):=\frac{1}{r^2}\int_{Q^{+}_{z,r}}|u(y,s)|^{3}dyds,\quad
M_{b}(r):=\frac{1}{r^2}\int_{Q^{+}_{z,r}}|b(y,s)|^{3}dyds,\quad
\]
\[
K_{b}(r):=\frac{1}{r^3} \int_{Q^{+}_{z,r}}|b(y,s)|^{2}dyds,
\]
\[
(u)_r(s):= \aint_{B^{+}_{x,r}}u(\cdot,s) dy ,\quad (b)_r(s):=
\aint_{B^{+}_{x,r}}b(\cdot,s) dy, \quad
(\pi)_{r}(s)=\aint_{B^{+}_{x,r}}\pi(y,s)dy,
\]
\[
G_{u,p,q}(r):=r^{1-\frac{3}{p}-\frac{2}{q}}\norm{u(y,s)}_{L^{p,q}_{y,s}(Q^{+}_{z,r})},
\quad
D_{u,\tilde{p},q}(r):=r^{2-\frac{3}{\tilde{p}}-\frac{2}{q}}\norm{\nabla
u(y,s)}_{L^{\tilde{p},q}_{y,s}(Q^{+}_{z,r})},
\]
\[
V_{u,\tilde{p},q}(r):=r^{2-\frac{3}{\tilde{p}}-\frac{2}{q}}
\norm{\omega(y,s)}_{L^{\tilde{p},q}_{y,s}(Q^{+}_{z,r})},\qquad
\omega=\nabla\times u,
\]
where $1 \leq p, q \leq \infty$, $3/\tilde{p}+2/q=3$ and
$1/p=1/\tilde{p}-1/3$,
\begin{equation*}
\tilde{Q}(r):=\frac{1}{r}\biggl(\int^{t}_{t-r^{2}}\Bigl(\int_{B^{+}_{x,r}}|\pi(y,s)-(\pi)_{r}(s)|^{\kappa^{*}}dy
\Bigr)^{\frac{\lambda}{\kappa^{*}}}ds\biggr)^{\frac{1}{\lambda}},
\end{equation*}
\begin{equation*}
Q(r):=\frac{1}{r}\biggl(\int^{t}_{t-r^{2}}\Bigl(\int_{B^{+}_{x,r}}|\pi(y,s)|^{\kappa^{*}}dy
\Bigr)^{\frac{\lambda}{\kappa^{*}}}ds\biggr)^{\frac{1}{\lambda}},
\end{equation*}
\begin{equation*}
Q_{1}(r):=\frac{1}{r}\biggl(\int^{t}_{t-r^2}\Bigl(\int_{B^{+}_{x,r}}
|\nabla\pi(y,s)|^{\kappa}dy\Bigr)^{\frac{\lambda}{\kappa}}ds\biggr)^{\frac{1}{\lambda}},
\end{equation*}
where $\kappa, \kappa^{*}$ and $\lambda$ are numbers satisfying
\begin{equation}\label{pq}
\frac{3}{\kappa}+\frac{2}{\lambda}=4,\quad\frac{1}{\kappa^{*}}
=\frac{1}{\kappa}-\frac{1}{3},
\quad 1<\lambda<2.
\end{equation}
 Next we recall suitable weak solutions for
the MHD equations \eqref{MHD} in three dimensions.

\begin{definition}\label{sws-3dnse}
Let $\Omega=\R^3_+$ and $Q_T=\R^3_+\times [0,T)$. A triple of
$(u,b,\pi)$ is a suitable weak solution to \eqref{MHD} if the
following conditions are satisfied:
\begin{itemize}
\item[(a)] The functions $u,b : Q_T\rightarrow \mathbb{R}^3$ and $\pi :
Q_T \rightarrow \mathbb{R}$ satisfy
\begin{equation*}
u,b\in L^{\infty}\big(I;L^{2}(\Omega)\big)\cap
L^{2}\big(I;W^{1,2}(\Omega)\big),\quad \pi\in
L^{\lambda}\big(I;L^{\kappa^*}(\Omega)\big),
\end{equation*}
\begin{equation*}
\nabla^{2}u,\nabla^{2}b\in
L^{\lambda}\big(I;L^{k}(\Omega)\big),\quad \nabla \pi\in
L^{\lambda}\big(I;L^{\kappa}(\Omega)\big),
\end{equation*}
where $\kappa,\kappa^*$ and $\lambda$ are numbers in \eqref{pq}.

\item[(b)] ($u,b,\pi$) solves the MHD equations in $Q_T$ in
the sense of distributions and $u$ and $b$ satisfy the boundary
conditions \eqref{noslip-slip} in the sense of traces.
\item[(c)] $u,b$ and $\pi$ satisfy the local energy inequality
\begin{equation*}
\int_{B^{+}_{x,r}}(\abs{u(x,t)}^2+\abs{b(x,t)}^2) \phi(x,t)
dx
\end{equation*}
\begin{equation*}
+2\int_{t_0}^t\int_{B^{+}_{x,r}} (\abs{\nabla
u(x,t')}^{2}+\abs{\nabla  b(x,t')}^{2}) \phi(x,t') dx dt'
\end{equation*}
\begin{equation*}
\leq  \int_{t_0}^t\int _{B^{+}_{x,r}}(\abs{u}^{2}+\abs{b}^{2})
(\partial_t\phi+\Delta \phi)dxdt'+ \int_{t_0}^t\int
_{B^{+}_{x,r}}\bke{ |u|^2 +|b|^2+ 2\pi} u\cdot\nabla \phi dxdt'
\end{equation*}
\begin{equation}\label{local-energy}
-2\int_{t_0}^t\int _{B^{+}_{x,r}}(b \cdot u)(b
\cdot\nabla\phi)dxdt'.
\end{equation}
for all $t\in I=(0,T)$ and for all nonnegative function $\phi \in
C_0^{\infty}(\R^3\times R)$.
\end{itemize}
\end{definition}

We consider the following Stokes system, which is the linearized
Navier-Stokes equations:
\begin{equation}\label{stokes-eqn}
v_t-\Delta v+\nabla p=f, \quad  {\rm{div}} \,v=0 \qquad \mbox{in
}\,\,Q_T:=\Omega\times (0,T)
\end{equation}
with initial data $v(x,0)=v_0(x)$. As in \eqref{noslip-slip},
boundary condition of $v$ is assumed to be no-slip, namely
$v(x,t)=0$ for $x\in\partial\Omega$. We recall maximal estimates of
the Stokes system in terms of mixed norms (see e.g. \cite[Theorem
5.1]{G86}).
\begin{lemma}\label{lem1}
Let $1<l,m<\infty$. Suppose that $f\in L^{l,m}_{x,t}(Q_T)$ and
$v_0\in D_l^{1-\frac{1}{m},m}$, where  $D_l^{1-\frac{1}{m},m}$ is a
Banach space with the following norm :
\begin{equation*}
D_l^{1-\frac{1}{m},m}(\Omega):= \ \bket{w \in L_{\sigma}^{l}(\Omega)
;
\norm{w}_{D_l^{1-\frac{1}{m},m}}=\norm{w}_{L^{l}}+\bke{\int_0^{\infty}
\norm{t^{\frac{1}{m}} A_l e^{-t A_l}w}^m_{L^{l}}\frac{dt}{t}
}^{\frac{1}{m}}<\infty },
\end{equation*}
where $A_l$ is the Stokes operator(see \cite{GS91} for the details).
If $(v,p)$ is the solution of the Stokes system \eqref{stokes-eqn}
with no-slip boundary conditions, then the following estimate is
satisfied:
\[
\norm{v_t}_{L^{l,m}_{x,t}(Q_T)}+\norm{\nabla^2
v}_{L^{l,m}_{x,t}(Q_T)}+\norm{\nabla p}_{L^{l,m}_{x,t}(Q_T)}
\]
\begin{equation}\label{stokes-estimate}
\leq
C\norm{f}_{L^{l,m}_{x,t}(Q_T)}+\norm{v_0}_{D_l^{1-\frac{1}{m},m}(\Omega)}.
\end{equation}
\end{lemma}



\section{Boundary regularity}

In this section, we prove a local regularity criterion for MHD
equations near the boundary and present the proofs of Theorem
\ref{main-thm-boundary-v} and \ref{main-thm-boundary-w}.
For simplicity, we write
$\Psi(r):=A_{u}(r)+A_{b}(r)+E_{u}(r)+E_{b}(r)$. Let $z=(x,
t)\in\Gamma\times I$ and from now on, without loss of generality, we
assume $x=0$ by translation. We first recall that the local energy
estimate.
\begin{equation}\label{local-energy-estimate}
\Psi(\frac{r}{2})\leq C\bigg(M^{\frac{2}{3}}_u(r)+K_b(r)+M_u(r)+
\frac{1}{r^2}\int_{Q^+_{z,r}}\abs{u}\abs{b}^2dz+\frac{1}{r^2}\int_{Q^+_{z,r}}\abs{u}\abs{\pi}dz\biggr).
\end{equation}

Next we prove a local regularity condition near boundary for MHD
equations (compare to \cite[Lemma 7]{GKT06} for the Navier-Stokes
equations).

\begin{proposition}\label{ep-regularity}
There exist $\epsilon^*>0$ and $r_0>0$ such that if $(u,b,\pi)$ is a
suitable weak solution of MHD equations satisfying Definition
\ref{sws-3dnse}, $z=(x,t) \in \Gamma \times I$, and
\begin{equation}
M_u(r)+M_b(r)+\tilde{Q}(r)<\epsilon^* \qquad \mbox{ for some }r\in
(0,r_0),
\end{equation}
then $z$ is regular point.
\end{proposition}


The proof of Proposition \ref{ep-regularity} is based on the
following lemma, which shows a decay property of $(u,b,\pi)$ in a
Lebesgue spaces. Although the method of proof is in principle the
similar as in \cite[Lemma 7, Lemma 8]{GKT06}, we present its details
for clarity (the proof will be given in Appendix).

\begin{lemma}\label{mhd-decay}
Let $0<\theta<\frac{1}{2}$. There exist $\epsilon_1>0$ and $r_*$
depending on $\lambda$ and $\theta$ such that if $(u,b,\pi)$ is a
suitable weak solution of the MHD equations satisfying Definition
\ref{sws-3dnse}, $z=(x,t)\in\Gamma\times(0,T)$, and
$M_{u}^{\frac{1}{3}}(r)+M_{b}^{\frac{1}{3}}(r)+\tilde{Q}(r)<\epsilon_1$
for some $r\in (0,r_*)$, then
\[ M_{u}^{\frac{1}{3}}(\theta
r)+M_{b}^{\frac{1}{3}}(\theta r)+\tilde{Q}(\theta r)<
C\theta^{1+\alpha}\bke{M_{u}^{\frac{1}{3}}(r)+M_{b}^{\frac{1}{3}}(r)
+\tilde{Q}(r)},
\]
where $0<\alpha<1$ and $C>0$ are constants.
\end{lemma}

Next lemma is estimates of the scaled integral of cubic term of $u$
and multiple of $u$ and square of $b$.

\begin{lemma}\label{estimate-ub}
Let $z=(x, t)\in\Gamma\times I$. Suppose that $u\in
L^{p,q}_{x,t}(Q^+_{z,r})$ with $3/p+2/q=2$, $3/2 \leq p \leq
\infty$. Then for $0<r<\rho/4$,
\begin{equation}\label{estimate-Mu}
M_{u}(r)\leq CG_{u,p,q}(r)\Psi(r) \leq
C\bigg(\frac{\rho}{r}\bigg)\Psi(\rho)G_{u,p,q}(r) ,
\end{equation}
\begin{equation}\label{estimate-Mb}
\frac{1}{r^2}\int_{Q^+_{z,r}}\abs{u}\abs{b}^2dz\leq
CG_{u,p,q}(r)\Psi(r) \leq
C\bigg(\frac{\rho}{r}\bigg)\Psi(\rho)G_{u,p,q}(r).
\end{equation}
\end{lemma}
\begin{proof}
It is sufficient to show estimate \eqref{estimate-Mb} because
\eqref{estimate-Mu} can be proved in the same way as
\eqref{estimate-Mb}. We note first that via H\"older's inequality
\begin{equation}\label{proof-est-mb}
\frac{1}{r^2}\int _{Q^+_{z,r}}\abs{u}\abs{b}^{2}dxds\leq
\frac{1}{r}\norm{u}_{L^{p,q}_{x,t}(Q^+_{z,r})}\frac{1}{r}\norm{b}^2_{L^{2p^{*},2q^{*}}_{x,t}(Q^+_{z,r})},
\end{equation}
where $p^*$ and $q^*$ are H\"older conjugates of $p$ and $q$. For
$\alpha:=(3-p^{*})/2p^{*}$ we see that
\[
\norm{b}_{L^{2p^{*}}_x(B^+_{x,r})}\leq
\norm{b}^{\alpha}_{L^2_{x}(B^+_{x,r})}\norm{b-(b)_r}^{1-\alpha}_{L^6_{x}(B^+_{x,r})}
+\norm{b}^{\alpha}_{L^2_{x}(B^+_{x,r})}\norm{(b)_r}^{1-\alpha}_{L^6_{x}(B^+_{x,r})}
\]
\[
\leq C\norm{b}^{\alpha}_{L^2_{x}(B^+_{x,r})}\norm{\nabla
b}^{1-\alpha}_{L^2_{x}(B^+_{x,r})}
+\norm{b}_{L^2_{x}(B^+_{x,r})}r^{-\frac{1}{2}+\frac{\alpha}{2}},
\]
where we used Poincar\'{e} inequality. Taking $L^{2q^*}$ norm in
temporal variable and using Young's inequality,
\[
\norm{b}^2_{L^{2p^{*},2q^{*}}_{x,t}(Q^+_{z,r})} \leq
C\norm{b}^{2}_{L^{2,\infty}_{x,t}(Q^+_{z,r})}+C\norm{\nabla
b}^{2}_{L^{2,2}_{x,t}(Q^+_{z,r})}.
\]
Recalling \eqref{proof-est-mb}, we can have
\[
\frac{1}{r^2}\int _{Q^+_{z,r}}\abs{u}\abs{b}^{2}dxds\leq
CG_{u,p,q}(r)\Psi(r) \leq C(\frac{\rho}{r})\Psi(\rho)G_{u,p,q}(r).
\]
This completes the proof.
\end{proof}

Next, we may continue with scaled norm of
$L^{2,2}_{x,t}(Q^{+}_{z_0,r})$ estimate of $b$.

\begin{lemma}\label{lem3.2}
Let $z=(x, t)\in\Gamma\times I$. Suppose that $u\in
L^{p,q}_{x,t}(Q^{+}_{z,r})$ with $3/p+2/q=2$ and $3/2\leq p <3$.
Then for $0<r<\rho/4$
\begin{equation}\label{boundary-b}
K_{b}(r)\leq
C\bigg(\frac{\rho}{r}\bigg)^3G_{u,p,q}^2(\rho)\Psi(\rho)
+C\bigg(\frac{r}{\rho}\bigg)^2K_{b}(\rho).
\end{equation}
\end{lemma}

\begin{proof}
For convenience, we write $x=(x_1, x_2, x_3)=(x',x_3)$ and by
translation, we assume that without loss of generality, $z=(0,0)\in
\Gamma\times I$. Let $\zeta(x,t)$ be a standard cut off function
supported in $Q_{\rho}$ such that $\zeta(x,t)=1$ in $Q_{\rho/2}$. We
set $g(x,t):= -\nabla\cdot([u\otimes b-b\otimes u]\zeta)$ in
$Q^+_{z, \rho}$ and we then define $\tilde{g}(x,t)$, an extension of
$g$ from $Q^+_{\rho}$ onto $Q_{\rho}$, in the following way:
$\tilde{g}(x,t)=g(x,t)$ if $x_3\geq 0$. On the other hand, if $x_3<
0$, then
\[
\tilde{g}_i(x',x_3,t)=g_i(x', -x_3, t),\qquad i=1,2
\]
\[
\tilde{g}_3(x',x_3,t)=-g_3(x', -x_3, t).
\]
This can be done by extending tangential components of $u$ and $b$
as even functions and normal components of $u$ and $b$ as odd
functions, respectively. We denote such extensions by $\tilde{u}$
and $\tilde{b}$ for simplicity. Here we also used the fact that
$\zeta$ and $\nabla' \zeta$ are even and $\partial_{x_3}\zeta$ is
odd with respect to $x_3-$variable, where $\nabla'=(\partial_{x_1},
\partial_{x_2})$.

Next, we define $\tilde{w}(x,t)$ for $(x,t)\in \R^3\times
(-\infty,0)$ by
\[
\tilde{w}(x,t)=\int_{-\infty}^{t}\int_{{\mathbb R}^3}
\frac{1}{(4\pi(t-s))^{\frac{3}{2}}}e^{-\frac{|x-y|^2}{4(t-s)}}\tilde{g}(y,s)
dyds,
\]
namely, $\tilde{w}$ satisfies
\[
\tilde{w}_t-\Delta \tilde{w}=\tilde{g} \qquad\mbox{in }\,\,
\R^3\times (-\infty,0).
\]
Moreover, we can see that $\partial_{x_3}\tilde{w}_i=0$ for $i=1,2$
and $\tilde{w}_3=0$ on $\{x_3=0\}$. Let $h=b-\tilde{w}$ in
$Q^+_{\rho}$. Then $h$ satisfies
\[
h_t-\Delta h=0\qquad \mbox{in} \ Q^{+}_{\frac{\rho}{2}}
\]
and $\partial_{x_3}h_i=0$ for $i=1,2$ and $h_3=0$ on $\{x_3=0\}\cap
Q_{\rho}$. Now we extend $h$ by the same manner as $g$, denoted by
$\tilde{h}$, from $Q^+_{\rho/2}$ onto $Q_{\rho/2}$. We then see that
\[
\tilde{h}_t-\Delta \tilde{h}=0 \qquad \mbox{in} \
Q_{\frac{\rho}{2}}.
\]
Via classical regularity theory, we have
\begin{equation}\label{tilde-h}
\int_{Q_r}|\tilde{h}|^2 dz \leq C(\frac{r}{\rho})^5
\int_{Q_{\frac{\rho}{2}}}|\tilde{h}|^2 dz.
\end{equation}
On the other hand, due to Sobolev embedding, we have $
\norm{\tilde{w}}_{L^2(B_{\rho})}\leq
\norm{\tilde{w}}_{L^2(\R^3)}\leq C\|\nabla
\tilde{w}\|_{L^{\frac{6}{5}}(\R^3)}$ and we then take $L^2$
integration for the above in time interval $(-\rho^2, 0)$ such that
we obtain

\[
\|\tilde{w}\|_{L^{2,2}_{x,t}(Q_{\rho})}\leq C\|\nabla
\tilde{w}\|_{L^{\frac{6}{5},2}_{x,t}(\R^3\times (-\rho^2,0))}\leq
C\|\tilde{u}\tilde{b}\zeta\|_{L^{\frac{6}{5},2}_{x,t}(\R^3\times
(-\rho^2,0))}
\]
\begin{equation}\label{tilde-w}
\leq C\|\tilde{u}\tilde{b}\|_{L^{\frac{6}{5},2}_{x,t}(Q_{\rho})}\leq
C\|\tilde{u}\|_{L^{p,q}_{x,t}(Q_{\rho})}\|\tilde{b}\|_{L^{\alpha,\beta}_{x,t}(Q_{\rho})},
\end{equation}
where $3/\alpha+2/\beta=3/2$ and $2\le \alpha < 6$, since $3/2\le
p<3$.
Using the estimate \eqref{tilde-w} and Sobolev inequality, we have
\[
\frac{1}{\rho^{3}}\|\tilde{w}\|^2_{L^{2,2}_{x,t}(Q_{\frac{\rho}{2}})}\leq
C\frac{1}{\rho^2}\|\tilde{u}\|^2_{L^{p,q}_{x,t}(Q_{\rho})}\frac{1}{\rho}\|\tilde{b}\|^2_{L^{\alpha,\beta}_{x,t}(Q_{\rho})}
\leq C\frac{1}{\rho^2}\|u\|^2_{L^{p,q}_{x,t}(Q^+_{\rho})}
\frac{1}{\rho}\|b\|^2_{L^{\alpha,\beta}_{x,t}(Q^+_{\rho})}
\]
\begin{equation}\label{estimate-bar_w1}
\leq
\frac{C}{\rho^2}\|u\|^2_{L^{p,q}_{x,t}(Q^+_{\rho})}(A_{b}(\rho)+E_{b}(\rho))
\leq CG^2_{u,p,q}(\rho)\Psi(\rho).
\end{equation}
Combining estimates \eqref{tilde-h} and \eqref{estimate-bar_w1}, we
obtain
\[
K_{b}(r)=\frac{1}{r^3}\norm{b}_{L^{2,2}_{x,t}(Q^+_r)}^2
\leq\frac{1}{r^3}\|\tilde{w}\|_{L^{2,2}_{x,t}(Q_r)}^2+\frac{1}{r^3}\|\tilde{h}\|_{L^{2,2}_{x,t}(Q_r)}^2
\]
\[
\leq C(\frac{\rho}{r})^3G_{u,p,q}^2(\rho)\Psi(\rho)
+C(\frac{r}{\rho})^2\frac{1}{\rho^3}\|b\|_{L^{2,2}_{x,t}(Q^+_{\frac{\rho}{2}})}^2
\]
\[
\leq C(\frac{\rho}{r})^3G_{u,p,q}^2(\rho)\Psi(\rho)
+C(\frac{r}{\rho})^2K_{b}(\rho).
\]
This completes the proof.
\end{proof}

In next lemma we show an estimate of the gradient of pressure
(compare to \cite[Lemma 11]{GKT06}).
\begin{lemma}\label{lem4.2}
Let $z=(x, t)\in\Gamma\times I$. Then for $0<r<\rho/4$,
\[
Q_{1}(r)\leq
C\bigg(\frac{\rho}{r}\biggr)\Big(A_u^{\frac{3-2\kappa}{2\kappa}}(\rho)
E_u^{\frac{1}{\lambda}}(\rho) +
A_b^{\frac{3-2\kappa}{2\kappa}}(\rho)
E_b^{\frac{1}{\lambda}}(\rho)\Big)
\]
\begin{equation}\label{estimate-local-Q}
+C\bigg(\frac{r}{\rho}\bigg)\Big(E_u^{\frac{1}{2}}(\rho)+Q_{1}(\rho)\Big),
\end{equation}
where $\kappa$ and $\lambda$ are numbers in \eqref{pq}.
\end{lemma}
\begin{proof}
We assume, via translation, that $z=(x, t)=(0, 0)$. We choose a
domain $\tilde{B}^{+}$ with a boundary such that
$B^{+}_{\frac{\rho}{2}} \subset \tilde{B}^{+} \subset B^{+}_{\rho}$,
and we denote $\tilde{Q}^{+}:=\tilde{B}^{+} \times (-\rho^2,0)$. Let
$(v,\pi_1)$ be the unique solution of the following the Stokes
system
\[
v_t-\Delta v+\nabla \pi_1=-(u\cdot \nabla)u+(b\cdot \nabla)b, \quad
\text{div}\, v=0 \ \ \text{in} \ \tilde{Q}^{+},
\]
\[
(\pi_1)_{\tilde{B}^{+}}=\aint_{\tilde{B}^{+}}\pi_1(y,t)dy=0, \quad
t\in(-\rho^2,0),
\]
\[
v=0 \quad \partial\tilde{B}^{+} \times [-\rho^2,0], \quad v=0 \quad
\tilde{B}^{+} \times \{t=-\rho^2\}.
\]
Using the Stokes estimate \eqref{stokes-estimate}, we have the
following estimate
\[
\frac{1}{\rho^2}\|v\|_{L_{x,t}^{\kappa,\lambda}(\tilde{Q}^{+})}
+\frac{1}{\rho}\|\nabla
v\|_{L_{x,t}^{\kappa,\lambda}(\tilde{Q}^{+})}
+\|v_t\|_{L_{x,t}^{\kappa,\lambda}(\tilde{Q}^{+})}+\|\nabla^{2}v\|_{L_{x,t}^{\kappa,\lambda}(\tilde{Q}^{+})}
\]
\[
+\frac{1}{\rho}\|\pi_{1}\|_{L_{x,t}^{\kappa,\lambda}(\tilde{Q}^{+})}
+\|\nabla \pi_{1}\|_{L_{x,t}^{\kappa,\lambda}(\tilde{Q}^{+})}
\]
\[
\leq
C\Big(\|(u\cdot\nabla)u\|_{L_{x,t}^{\kappa,\lambda}(\tilde{Q}^{+})}+\|(b\cdot\nabla)b\|_{L_{x,t}^{\kappa,\lambda}(\tilde{Q}^{+})}
\Big)
\]
\[
\leq
C\Big(\|(u\cdot\nabla)u\|_{L_{x,t}^{\kappa,\lambda}(Q_{\rho}^{+})}
+\|(b\cdot\nabla)b\|_{L_{x,t}^{\kappa,\lambda}(Q_{\rho}^{+})}\Big)
\]
\[
\leq C\big(\rho A_u^{\frac{3-2\kappa}{2\kappa}}(\rho)
E_u^{\frac{1}{\lambda}}(\rho) +\rho
A_b^{\frac{3-2\kappa}{2\kappa}}(\rho)
E_b^{\frac{1}{\lambda}}(\rho)\big),
\]
where we used the following estimates in last inequality above:
\[
\norm{(u\cdot \nabla)u}_{L^{\kappa,\lambda}_{x,t}(Q_{\rho}^{+})}
+\norm{(b\cdot \nabla)b}_{L^{\kappa,\lambda}_{x,t}(Q_{\rho}^{+})}
\]
\[
\leq
\norm{u}^{\frac{3-2\kappa}{\kappa}}_{L^{2,\infty}_{x,t}(Q_{\rho}^{+})}\norm{\nabla
u}^{\frac{2}{\lambda}}_{L^{2,2}_{x,t}(Q_{\rho}^{+})}
+\norm{b}^{\frac{3-2\kappa}{\kappa}}_{L^{2,\infty}_{x,t}(Q_{\rho}^{+})}\norm{\nabla
b}^{\frac{2}{\lambda}}_{L^{2,2}_{x,t}(Q_{\rho}^{+})}
\]
\[
\leq C\rho
A_u^{\frac{3-2\kappa}{\kappa}}(\rho)E^{\frac{2}{\lambda}}_u(\rho)
+C\rho
A_b^{\frac{3-2\kappa}{\kappa}}(\rho)E^{\frac{2}{\lambda}}_b(\rho).
\]
 Next, let $w=u-v$ and
 $\pi_2=\pi-(\pi)_{B_{\frac{\rho}{2}}^{+}}-\pi_1$. Then
 $(w,\pi_2)$ solves the following the boundary value problem:
\[
w_t-\Delta w+\nabla \pi_2=0, \quad \text{div}\, w=0 \qquad\text{in}
\ \tilde{Q}^{+},
\]
\[
w=0 \quad \ \text{on}\ (\partial\tilde{B}^{+}\cap\{x_3=0\}) \times
[-\rho^2,0].
\]
Now we take $\kappa'$ such that $3/\kappa'+2/\lambda=2$. Then from
the local estimate near the boundary for the Stokes systems (see
\cite{GAS02}), we obtain
\[
\|\nabla^{2}w\|_{L_{x,t}^{\kappa',\lambda}(Q^{+}_{\frac{\rho}{4}})}
+\|\nabla\pi_{2}\|_{L_{x,t}^{\kappa',\lambda}(Q^{+}_{\frac{\rho}{4}})}\\
\]
\[
\leq\frac{C}{\rho^{2}}\bigg(\frac{1}{\rho^{2}}\|w\|_{L_{x,t}^{\kappa,\lambda}(Q^{+}_{\frac{\rho}{2}})}
+\frac{1}{\rho}\|\nabla
w\|_{L_{x,t}^{\kappa,\lambda}(Q^{+}_{\frac{\rho}{2}})}
+\frac{1}{\rho}\|\pi_{2}\|_{L_{x,t}^{\kappa,\lambda}(Q^{+}_{\frac{\rho}{2}})}\bigg)
\]
\[\leq\frac{C}{\rho^{2}}\bigg(\frac{1}{\rho}\|\nabla u\|_{L_{x,t}^{\kappa,\lambda}(Q^{+}_{\frac{\rho}{2}})}
+\|\nabla\pi\|_{L_{x,t}^{\kappa,\lambda}(Q^{+}_{\frac{\rho}{2}})}
+\frac{1}{\rho}\|\nabla
v\|_{L_{x,t}^{\kappa,\lambda}(Q^{+}_{\frac{\rho}{2}})}
+\frac{1}{\rho}\|\pi_{1}\|_{L_{x,t}^{\kappa,\lambda}(Q^{+}_{\frac{\rho}{2}})}\bigg),
\]
where Poincar\'e-Sobolev inequality is used. Since $\|\nabla
u\|_{L_{x,t}^{\kappa,\lambda}(Q^{+}_{\rho})}\leq
C\rho^{2}E_u^{\frac{1}{2}}(\rho)$, we have
\[
\|\nabla\pi_{2}\|_{L_{x,t}^{\kappa',\lambda}(Q^{+}_{\frac{\rho}{4}})}\leq
\frac{C}{\rho^2}\bke{\rho E_u^{\frac{1}{2}}(\rho)+\rho Q_{1}(\rho)
+\rho A_u^{\frac{3-2\kappa}{2\kappa}}(\rho)
E_u^{\frac{1}{\lambda}}(\rho) +\rho
A_b^{\frac{3-2\kappa}{2\kappa}}(\rho) E_b^{\frac{1}{\lambda}}(\rho)}
\]
\[
=\frac{C}{\rho}\bigl(E_u^{\frac{1}{2}}(\rho)+ Q_{1}(\rho)+
A_u^{\frac{3-2\kappa}{2\kappa}}(\rho) E_u^{\frac{1}{\lambda}}(\rho)
+ A_b^{\frac{3-2\kappa}{2\kappa}}(\rho)
E_b^{\frac{1}{\lambda}}(\rho)\bigr).
\]
Let $0\leq r\leq\rho/4$. Noting that
$\|\nabla\pi_{2}\|_{L_{x,t}^{\kappa,\lambda}(Q^{+}_{r})}\leq
Cr^{2}\|\nabla\pi_{2}\|_{L_{x,t}^{\kappa',\lambda}(Q^{+}_{r})}$, we
have
\begin{equation*}
Q_{1}(r)=\frac{1}{r}\|\nabla\pi\|_{L_{x,t}^{\kappa,\lambda}(Q^{+}_{r})}
\leq\frac{1}{r}\bigl(\|\nabla\pi_{1}\|_{L_{x,t}^{\kappa,\lambda}(Q^{+}_{r})}
+\|\nabla\pi_{2}\|_{L_{x,t}^{\kappa,\lambda}(Q^{+}_{r})}\bigr)
\end{equation*}
\begin{equation*}
\leq\frac{1}{r}\bigl(\|\nabla\pi_{1}\|_{L_{x,t}^{\kappa,\lambda}(Q^{+}_{\rho})}
+r^2\|\nabla\pi_{2}\|_{L_{x,t}^{\kappa',\lambda}(Q^{+}_{r})}\bigr)
\end{equation*}
\begin{equation*}
\leq C(\frac{\rho}{r})\bigl(
A_u^{\frac{3-2\kappa}{2\kappa}}(\rho)E^{\frac{1}{\lambda}}_u(\rho)+
A_b^{\frac{3-2\kappa}{2\kappa}}(\rho)E^{\frac{1}{\lambda}}_b(\rho)\bigr)
\end{equation*}
\begin{equation*}
\quad +C(\frac{r}{\rho})\bigl(E_u^{\frac{1}{2}}(\rho)+ Q_{1}(\rho)+
A_u^{\frac{3-2\kappa}{2\kappa}}(\rho) E_u^{\frac{1}{\lambda}}(\rho)
+ A_b^{\frac{3-2\kappa}{2\kappa}}(\rho)
E_b^{\frac{1}{\lambda}}(\rho)\bigr)
\end{equation*}
\begin{equation*}
\leq C(\frac{\rho}{r})\Big(A_u^{\frac{3-2\kappa}{2\kappa}}(\rho)
E_u^{\frac{1}{\lambda}}(\rho) +
A_b^{\frac{3-2\kappa}{2\kappa}}(\rho)
E_b^{\frac{1}{\lambda}}(\rho)\Big)
+C(\frac{r}{\rho})\Big(E_u^{\frac{1}{2}}(\rho)+Q_{1}(\rho)\Big).
\end{equation*}
This completes the proof.
\end{proof}
We remark that, via Young's inequality, \eqref{estimate-local-Q} can
be estimated as follows:
\begin{equation}\label{est-local-Q1}
Q_{1}(r)\leq
C\bke{\bigg(\frac{\rho}{r}\bigg)+\bigg(\frac{r}{\rho}\bigg)}\Psi(\rho)
+C\bke{\frac{r}{\rho}}\bke{Q_{1}(\rho)+1}.
\end{equation}

Next lemma shows an estimate of a scaled norm of pressure.
\begin{lemma}\label{lem4.3}
Let $z=(x, t)\in\Gamma\times I$. Suppose that $\nabla \pi \in
L^{\kappa,\lambda}_{x,t}(Q_{\rho})$ and $\pi \in
L^{\kappa^*,\lambda}_{x,t}(Q_{\rho})$, where $3/\kappa+2/\lambda=4$,
$1/\kappa^*=1/\kappa-1/3$ and $1 < \lambda < 2$. Then for
$0<r<\rho/4$,
\begin{equation}\label{boundary-pressure}
Q(r)\leq
C\bke{\frac{\rho}{r}}Q_1(\rho)+C\bke{\frac{r}{\rho}}^{\frac{3}{\kappa^*}-1}Q(\rho).
\end{equation}
\end{lemma}
\begin{proof}
Since $1<\lambda<2$, we note that $\frac{3}{2}<\kappa^*<3$. We first
observe that due to H\"{o}lder inequality
\[
\|(\pi)_{\rho }\|_{L_x^{\kappa^*}(B^+_{x,r})}\leq
C(\frac{r}{\rho})^{\frac{3}{\kappa^*}}\|\pi\|_{L_x^{\kappa}(B^+_{x,\rho})}.
\]
Therefore, due to Poincar\'e-Sobolev inequality, we have
\[
\norm{\pi}_{L^{\kappa^*,\lambda}_{x,t}(Q^+_{z,r})}\leq
\norm{\pi-(\pi)_{\rho}}_{L^{\kappa^*,\lambda}_{x,t}(Q^+_{z,r})}
+\norm{(\pi)_{\rho}}_{L^{\kappa^*,\lambda}_{x,t}(Q^+_{z,r})}
\]
\[
\leq C\|\nabla \pi\|_{L^{\kappa,\lambda}_{x,t}(Q^+_{z,r})}
+C(\frac{r}{\rho})^{\frac{3}{\kappa^*}}
\|\pi\|_{L^{\kappa^*,\lambda}_{x,t}(Q^+_{z,\rho})}.
\]
Dividing both sides by $r$, we have
\[
\frac{1}{r}\|\pi\|_{L^{\kappa^*,\lambda}_{x,t}(Q^+_{z,r})}\leq
C(\frac{\rho}{r})\frac{1}{\rho}\|\nabla
\pi\|_{L^{\kappa,\lambda}_{x,t}(Q^+_{z,\rho})}
+C(\frac{r}{\rho})^{\frac{3}{\kappa^*}-1}\frac{1}{\rho}
\|\pi\|_{L^{\kappa^*,\lambda}_{x,t}(Q^+_{z,\rho})}.
\]
This completes the proof.
\end{proof}

We are ready to present the proof of Theorem
\ref{main-thm-boundary-v}.
\begin{pfthm1}
We note first that via H\"older's inequality, it suffices to show
the case that $3/p+2/q=2$, $2<q<\infty$. Recalling Lemma
\ref{lem4.2} and Lemma \ref{lem4.3}, we have
\[
\frac{1}{r^2}\int_{Q^+_{z,r}}\abs{u}\abs{\pi} dz \leq
\frac{1}{r}\norm{u}_{L^{p,q}_{x,t}(Q^+_{z,r})}\frac{1}{r}\norm{\pi}_{L^{\kappa^{*},\lambda}_{x,t}(Q^+_{z,r})}
\]
\[
\leq C\epsilon\Bigg( (\frac{\rho}{r})Q_1(\frac{\rho}{2})
+(\frac{r}{\rho})^{\frac{3}{\kappa^{*}}-1}Q(\frac{\rho}{2})\Biggr)
\]
\begin{equation}\label{estimate-u-b}
\leq C\epsilon \bke{(\frac{\rho}{r})^2+1}\Psi(\rho) +C\epsilon
\bke{Q_{1}(\rho)+1}
+C\epsilon(\frac{r}{\rho})^{\frac{3}{\kappa^{*}}-1}Q(\rho),
\end{equation}
where \eqref{est-local-Q1} is also used. With aid of Lemma
\ref{estimate-ub}, Lemma \ref{estimate-Mu}, \eqref{est-local-Q1} and
\eqref{estimate-u-b}, we have
\begin{equation*}
\Psi(\frac{r}{2}) \leq
C\epsilon^{\frac{2}{3}}(\frac{\rho}{r})^{\frac{2}{3}}\Psi^{\frac{2}{3}}(\rho)
+C\epsilon(\frac{\rho}{r})\Psi(\rho)
+C\epsilon^2(\frac{\rho}{r})^3\Psi(\rho)+C(\frac{r}{\rho})^2K_{b}(\rho)
\end{equation*}
\begin{equation*}
+C\epsilon \big((\frac{\rho}{r})^2+1)\big)\Psi(\rho) +C\epsilon
(Q_{1}(\rho)+1)
+C\epsilon(\frac{r}{\rho})^{\frac{3}{\kappa^{*}}-1}Q(\rho)
\end{equation*}
\begin{equation*}
\leq
C\bkt{\epsilon^2(\frac{\rho}{r})^3+\epsilon(\frac{\rho}{r})^2+(\epsilon^{\frac{1}{2}}+\epsilon)
(\frac{\rho}{r})+\epsilon+(\frac{r}{\rho})^2}\Psi(\rho)
\end{equation*}
\begin{equation}\label{estimate-psi}
+C\epsilon Q_{1}(\rho)+C\epsilon
+C\epsilon(\frac{r}{\rho})^{\frac{3}{\kappa^{*}}-1}Q(\rho),
\end{equation}
where we used the Young's inequality and $K_b(\rho)\leq \Psi(\rho)$.
Let $\epsilon_1$ and $\epsilon_2$ be small positive numbers, which
will be specified later.

Adding $\epsilon_1 Q_1(\frac{r}{2})$ and $\epsilon_2 Q(\frac{r}{2})$
to both sides in \eqref{estimate-psi}, and using
\eqref{est-local-Q1} and Lemma \ref{lem4.3}, we obtain
\begin{equation*}
\Psi(\frac{r}{2})+\epsilon_1Q_1(\frac{r}{2})+\epsilon_2Q(\frac{r}{2})
\end{equation*}
\begin{equation*}
\leq
C\bkt{\epsilon^2\bke{\frac{\rho}{r}}^3+\epsilon(\frac{\rho}{r})^2+(\epsilon^{\frac{1}{2}}+\epsilon)
(\frac{\rho}{r})+\epsilon+(\frac{r}{\rho})^2}\Psi(\rho)
\end{equation*}
\begin{equation*}
+C\epsilon Q_{1}(\rho)+C\epsilon
+C\epsilon(\frac{r}{\rho})^{\frac{3}{\kappa^{*}}-1}Q(\rho)
+C\epsilon_1\bkt{(\frac{\rho}{r})+(\frac{r}{\rho})}\Psi(\rho),
\end{equation*}
\begin{equation*}
+C\epsilon_1(\frac{r}{\rho})(Q_{1}(\rho)+1)
+C\epsilon_2(\frac{\rho}{r})Q_1(\rho)+C\epsilon_2(\frac{r}{\rho}
)^{\frac{3}{\kappa^{*}}-1}Q(\rho)
\end{equation*}
\begin{equation*}
\leq
C\bkt{\epsilon^2(\frac{\rho}{r})^3+\epsilon(\frac{\rho}{r})^2+(\epsilon^{\frac{1}{2}}+\epsilon+\epsilon_1)
(\frac{\rho}{r})+\epsilon+(\frac{r}{\rho})^2+\epsilon_1(\frac{r}{\rho})}\Psi(\rho)
\end{equation*}
\begin{equation*}
+C\bkt{\epsilon+\epsilon_1(\frac{r}{\rho})+\epsilon_2(\frac{\rho}{r})}
Q_{1}(\rho)+C(\epsilon+\epsilon_2)\frac{r}{\rho}^{\frac{3}{\kappa^{*}}-1}Q(\rho)
+C\bkt{\epsilon+\epsilon_1(\frac{r}{\rho})}.
\end{equation*}
We fix $\theta \in (0,\frac{1}{4})$ with
$C(\theta+\theta^{\frac{3}{\kappa^{*}}-1})<\frac{1}{4}$ and then
choose $\epsilon_1, \epsilon_2$ and $\epsilon$ satisfying
\[
0<\epsilon_1<\frac{\theta}{16C},\qquad
0<\epsilon_2<\frac{\epsilon_1\theta}{8C},\qquad
0<\epsilon<\min\bket{\frac{\epsilon^*}{16C}, \,\,\epsilon_2,\,\,
\frac{\theta^6}{16C^2}}.
\]
Therefore, we have
\begin{equation}\label{K1K}
\Psi(\theta r)+\epsilon_1 Q_1(\theta r)+\epsilon_2 Q(\theta r) \leq
\frac{\epsilon^*}{8}+\frac{1}{2}\Big(\Psi(r)+\epsilon_1
Q_1(r)+\epsilon_2 Q( r) \Big).
\end{equation}
Iterating \eqref{K1K}, we can see that there exists a sufficiently
small $r_0>0$ such that for all $r<r_0$
\[
\Psi(r)+\epsilon_1 Q_1(r)+\epsilon_2 Q(r) \leq \frac{\epsilon^*}{4}.
\]
Therefore, we conclude that $\Psi(r)\leq \epsilon^*/8$. Next, we use
the estimates \eqref{est-local-Q1} and \eqref{boundary-pressure} to
obtain that there is $r_1>0$ such that $Q(r)\leq \epsilon^*/4$ for
all $r<r_1$. This can be shown by the method of iterations as in
\eqref{K1K}. Summing up, we obtain $\Psi(r)+Q(r)\leq \epsilon^*/2$
for all $r<r_1$, which implies the regularity condition in
Proposition \ref{ep-regularity}. This completes the proof.
\end{pfthm1}


The proof of Theorem \ref{main-thm-boundary-w} is given below.
\begin{pfthm2}
As mentioned earlier, it suffices to show the case that $3/p+2/q=3$,
$2<q<\infty$. We first show that the gradient of velocity is
controlled by vorticity. To be more precise, we prove the following
estimate (compare to \cite[Lemma 3.6]{GKT07}):
\begin{equation}\label{boundary-vorticity-estimate1}
D_{u,\tilde{p},q}(r)\leq
C(\frac{\rho}{r})V_{u,\tilde{p},q}(\rho)+C(\frac{r}{\rho}
)^{\frac{3}{\tilde{p}}-1}D_{u,\tilde{p},q}(\rho).
\end{equation}
Indeed, let $\xi$ be a cut off supported in $Q_{\rho}$ and $\xi=1$
in $Q_{\frac{\rho}{2}}$. We consider
\begin{equation*}
\quad \left\{
\begin{array}{ll}
\displaystyle -\Delta v =\nabla \times ( \omega \xi) \quad \text{in}\quad \mathbb{R}^3_{+}\\
\vspace{-3mm}\\
\displaystyle v =0 \quad \text{on}\quad \{x_3=0\}
\end{array}\right.
\end{equation*}
and set $h=u-v$. So $h$ is harmonic function in
$B^{+}_{\frac{r}{2}}$ with $ h=0$ on $\partial
B^{+}_{\frac{r}{2}}\cap\{x_3=0\}.$ By the mean value theorem of
harmonic functions and $L^p$ estimates of elliptic equations for
each fixed time $t$
\[
\|\nabla h\|_{L_x^{\tilde{p}}(B^{+}_r)}\leq C(\frac{r}{\rho}
)^{\frac{3}{\tilde{p}}}\|\nabla
h\|_{L_x^{\tilde{p}}(B^{+}_{\frac{\rho}{2}})} \leq
C(\frac{r}{\rho})^{\frac{3}{\tilde{p}}}\bke{\|\nabla
u\|_{L_x^{\tilde{p}}(B^{+}_{\rho})}+\|\nabla
v\|_{L_x^{\tilde{p}}(B^{+}_{\rho})}}
\]
\[
\leq C(\frac{r}{\rho} )^{\frac{3}{\tilde{p}}}\bke{\|\nabla
u\|_{L_x^{\tilde{p}}(B^{+}_{\rho})}+\|\ \omega
\|_{L_x^{\tilde{p}}(B^{+}_{\rho})}}.
\]
Adding together above estimates,
\[
\|\nabla u \|_{L_x^{\tilde{p}}(B^{+}_{r})}\leq \|\nabla v
\|_{L_x^{\tilde{p}}(B^{+}_{r})}+\|\nabla h
\|_{L_x^{\tilde{p}}(B^{+}_{r})}\leq C\|\ \omega
\|_{L_x^{\tilde{p}}(B^{+}_{r})}+C(\frac{r}{\rho}
)^{\frac{3}{\tilde{p}}}\|\nabla u\|_{L_x^{\tilde{p}}(B^{+}_{\rho})}.
\]
Taking $L^{q}$-norm in time and dividing both sides by $r$, we
obtain \eqref{boundary-vorticity-estimate1}. Via the method of
iteration, the estimate implies that the scaled norm of gradient of
velocity becomes sufficiently small. Since argument is
straightforward, we skip its details. We deduce the Theorem via
Corollary \ref{cor-boundary-v}.
\end{pfthm2}

\section{Appendix}
In this Appendix we present the proof of Lemma \ref{mhd-decay} and
interior regularity is compared to boundary regularity given in
Theorem \ref{main-thm-boundary-v}.

\subsection{Proof of Lemma \ref{mhd-decay}}

As mentioned earlier, the method of proof is quite similar to that
of \cite[Lemma 8]{GKT06} and main difference is mostly caused by the
presence of magnetic field $b$. Therefore, we give the mainstream of
the proof, instead giving all the details.

\begin{pf-mhd-decay}
For convenience, we denote
$\phi(r):=M_{u}^{\frac{1}{3}}(r)+M_{b}^{\frac{1}{3}}(r)+\tilde{Q}(r)$.
Suppose the statement is not true. So for any $\alpha\in(0,1)$ and
$C>0$, there exist $z_{n}=(x_{n},t_{n})$, $r_{n}\searrow 0$ and
$\epsilon_{n}\searrow 0$ such that
\begin{equation*}
\phi(r_{n})=\epsilon_{n},\qquad \phi(\theta
r_n)>C\theta^{1+\alpha}\phi(r_{n})=C\theta^{1+\alpha}\epsilon_{n}.
\end{equation*}
Let $w=(y,s)$ where $y=\frac{(x-x_{n})}{r_{n}}$,
$s=\frac{(t-t_{n})}{r^{2}_{n}}$ and we define $v_{n}, b_{n}$ and $
\pi_{n}$ as follows:
\[
v_{n}(w)=\frac{r_{n}}{\epsilon_{n}}u(z),\quad
b_{n}(w)=\frac{r_{n}}{\epsilon_{n}}b(z),\quad
\pi_{n}(w)=\frac{r^{2}_{n}}{\epsilon_{n}}(\pi(z)-(\pi)_{B^{+}_{r_n}}(z)).
\]
We also introduce some scaling invariant functionals defined by
\begin{equation*}
T_{u}(v_{n},\theta):=\frac{1}{\theta^{2}}\int_{Q^{+}_{\theta}}|v_{n}|^{3}dw,\quad
T_{b}(b_{n},\theta):=\frac{1}{\theta^{2}}\int_{Q^{+}_{\theta}}|b_{n}|^{3}dw,\quad
\end{equation*}
\begin{equation*}
P_{1}(\pi_{n},\theta):=\frac{1}{\theta}\biggl(\int^{0}_{-\theta^{2}}
\Bigl(\int_{B^{+}_{\theta}}|\nabla
\pi_{n}|^{\kappa}dy\Bigr)^{\frac{\lambda}{\kappa}}ds
\biggr)^{\frac{1}{\lambda}},
\end{equation*}
\begin{equation*}
\tilde{P}(\pi_{n},\theta):=\frac{1}{\theta}\biggl(\int^{0}_{-\theta^{2}}
\Bigl(\int_{B^{+}_{\theta}}|\pi_{n}-(\pi_{n})_{B^{+}_{\theta}}|^{\kappa^{*}}dy
\Bigr)^{\frac{\lambda}{\kappa^{*}}}ds\biggr)^{\frac{1}{\lambda}},
\end{equation*}
where $\kappa^{*}$, $\kappa$ and $\lambda$ are numbers in
\eqref{pq}.  Let
$\tau_{n}(\theta)=T_u^{\frac{1}{3}}(v_{n},\theta)+T_b^{\frac{1}{3}}(b_{n},\theta)
+\tilde{P}(\pi_{n},\theta)$. The change of variables lead to
\begin{equation}\label{scaling-1000}
\tau_{n}(1)=\|v_{n}\|_{L_{x,t}^{3,3}(Q^{+}_{1})}+\|b_{n}\|_{L_{x,t}^{3,3}(Q^{+}_{1})}
+\|\pi_{n}\|_{L_{x,t}^{\kappa^{*},\lambda}(Q^{+}_{1})}=1, \quad
\tau_{n}(\theta)\geq C\theta^{1+\alpha}.
\end{equation}
On the other hand, $v_{n}, b_n$ and $ \pi_{n}$ solve the following
system in a weak sense:
\begin{equation*}
\begin{cases}
\ \partial_{s}v_{n} -\Delta
v_{n}+\epsilon_{n}(v_{n}\cdot\nabla)v_{n}-\epsilon_{n}(b_{n}\cdot\nabla)b_{n}
+\nabla\pi_{n}=0,\quad\text{div}\ v_{n}=0&\quad\mbox{ in }\
Q^{+}_{1},
\\
\ \partial_{s}b_{n} -\Delta
b_{n}+\epsilon_{n}(v_{n}\cdot\nabla)b_{n}-\epsilon_{n}(b_{n}\cdot\nabla)u_{n}=0,\quad\text{div}\
b_{n}=0&\quad\mbox{ in }\ Q^{+}_{1},
\end{cases}
\end{equation*}
with boundary data $v_n=0$, $b_n\cdot \nu=0$ and $(\nabla \times
b_n)\times \nu=0$ on $B_1\cap\{x_{3}=0\}\times(-1,0)$. Since
$\tau_{n}(1)=1$, we have following weak convergence:
\begin{equation*}
v_{n}\rightharpoonup u \quad\text{in}\
L_{x,t}^{3,3}(Q^{+}_{1}),\qquad\ b_{n}\rightharpoonup b
\quad\text{in}\ L_{x,t}^{3,3}(Q^{+}_{1}),\qquad\
\pi_{n}\rightharpoonup \pi\quad\text{in}\
L_{x,t}^{\kappa^*,\lambda}(Q^{+}_{1}),
\end{equation*}
and $(\pi)_{B^{+}_{1}}(s)=0$. Moreover, we note that
$\partial_{s}v_{n}$ and $\partial_{s}b_{n}$ are uniformly bounded in
$L^{\lambda}\big((-1,0);(W^{2,2}(B^{+}_{1}))'\big)$, respectively
and we also have
\begin{equation}\label{timeprimevn}
\partial_{s}v_{n}\rightharpoonup\partial_{s}u,\quad\partial_{s}b_{n}\rightharpoonup\partial_{s}b
\qquad\text{in}\ L^{\lambda}\big((-1,0);(W^{2,2}(B^{+}_{1}))'\big),
\end{equation}
Using the local energy inequality \eqref{local-energy}, $\nabla
v_{n}$ and $\nabla b_{n}$ are uniformly bounded in
$L_{x,t}^{2,2}(Q^{+}_{3/4})$, which implies
\begin{equation}\label{gradient-l2-weak}
v_{n} \rightharpoonup u,\quad  b_{n} \rightharpoonup
b\qquad\text{in}\ W^{1,2}(Q^{+}_{3/4}).
\end{equation}
Its verification is rather standard, we skip its details (compare to
\cite[Lemma 8]{GKT06}).

We note that $u$, $b$ and $\pi$ solve the following linear Stokes
system
\begin{equation*}
\begin{cases}
\partial_{s}u-\Delta u+\nabla \pi=0,\quad\text{div}\
u=0\quad\text{in}\ Q^{+}_{1},
\\
\partial_{s}b-\Delta b=0,\quad\text{div}\
b=0\quad\text{in}\ Q^{+}_{1},
\end{cases}
\end{equation*}
with boundary data $u=0$, $b\cdot \nu=0$ and $(\nabla \times
b)\times \nu=0$ on $B_1\cap\{x_{3}=0\}\times(-1,0)$. We can show
that
\begin{equation}
\partial_{s}v_n, \ \partial_{s}b_n, \,\Delta v_n,
\, \Delta b_n, \ \nabla \pi_n \rightharpoonup \partial_{s} u, \,
\partial_{s} b,
\ \Delta u, \, \Delta b, \nabla \pi \quad\mbox{ in
}\,L_{x,t}^{\kappa,\lambda}(Q^{+}_{5/8}).
\end{equation}
Indeed, due to H\"older inequality, we see that
\[
\norm{\abs{(v_{n}\cdot\nabla) v_{n}}+\abs{(b_{n}\cdot\nabla)
b_{n}}+\abs{(v_{n}\cdot\nabla) b_{n}}+\abs{(b_{n}\cdot\nabla)
v_{n}}}_{L_{x,t}^{\kappa,\lambda}(Q^{+}_{3/4})}
\]
\begin{equation}\label{estimate-500}
\leq C(\|\nabla
v_{n}\|^{\frac{2}{\lambda}}_{L^{2,2}_{x,t}(Q^{+}_{3/4})}+\|\nabla
b_{n}\|^{\frac{2}{\lambda}}_{L^{2,2}_{x,t}(Q^{+}_{3/4})})
(\|v_{n}\|^{\frac{3-2\kappa}{\kappa}}_{L_{x,t}^{2,\infty}(Q^{+}_{3/4})}+
\|b_{n}\|^{\frac{3-2\kappa}{\kappa}}_{L_{x,t}^{2,\infty}(Q^{+}_{3/4})}).
\end{equation}
Using the local estimates of Stokes system and heat equations near
boundary,
\[
\|\abs{\partial_{s}v_{n}}+\abs{\partial_{s}b_{n}}+\abs{\nabla^{2}v_{n}}
+\abs{\nabla^{2}b_{n}}+\abs{\nabla
\pi_{n}}\|_{L_{x,t}^{\kappa,\lambda}(Q^{+}_{5/8})}
\]
\[
\leq C\|\abs{v_{n}}+\abs{b_{n}}+\abs{\nabla v_{n}}+\abs{\nabla
b_{n}}+\abs{\pi_{n}}\|_{L_{x,t}^{\kappa,\lambda}(Q^{+}_{3/4})}
\]
\[
+C\epsilon_{n}(\norm{\abs{(v_{n}\cdot\nabla)
v_{n}}+\abs{(b_{n}\cdot\nabla) b_{n}}+\abs{(v_{n}\cdot\nabla)
b_{n}}+\abs{(b_{n}\cdot\nabla)
v_{n}}}_{L_{x,t}^{\kappa,\lambda}(Q^{+}_{3/4})}).
\]
We note that, due to \eqref{estimate-500}, the righthand side of the
above estimate is bounded by $C(1+\epsilon_{n})$.

According to estimates of the perturbed stokes system near boundary
in \cite{SSS06}, $u, b$ are H\"{o}lder continuous in $Q^{+}_{1/2}$
with the exponent $\alpha$ with $0<\alpha<2(1-1/\lambda)$. Here we
fix $\alpha_{0}=1-1/\lambda$. Then, by H\"{o}lder continuity of $u,
b$ and strong convergence of the $L^{3}-$norm of $v_{n}, b_n$, we
obtain
\begin{equation}\label{scaling-2000}
T_u(v_{n},\theta) \rightarrow T_u(u,\theta),\,\,\,T_b(b_{n},\theta)
\rightarrow T_b(b,\theta),\quad
T_u^{\frac{1}{3}}(u,\theta)+T_b^{\frac{1}{3}}(b,\theta)\leq
C_1\theta^{1+\alpha_{0}}.
\end{equation}

Next we need to estimate $\tilde{P}(\pi_n, \theta)$. Let
$\tilde{B}^{+}$ be a domain with smooth boundary such that
$B^{+}_{11/16}\subset\tilde{B}^{+}\subset B^{+}_{3/4}$, and
$\tilde{Q}^{+}:=\tilde{B}^{+}\times(-(3/4)^2,0)$. Now we consider
the following initial and boundary problem of $\bar{v}_{n},
\bar{b}_{n}, \bar{\pi}_{n}$
\begin{equation*}
\partial_{s}\bar{v}_{n}-\Delta\bar{v}_{n}+\nabla\bar{\pi}_{n}
=-\epsilon_{n}(\bar{v}_{n}\cdot\nabla)\bar{v}_{n}+\epsilon_{n}(\bar{b}_{n}\cdot\nabla)\bar{b}_{n},\quad
{\rm{div}}\,\bar{v}_{n}=0\qquad \text{in}\ \tilde{Q}^{+},
\end{equation*}
\begin{equation*}
(\bar{\pi}_{n})_{\tilde{B}^{+}}(s)=0,\quad s\in
(-(\frac{3}{4})^2,0),
\end{equation*}
\begin{equation*}
\bar{v}_{n}=0\quad\text{on}\
\partial\tilde{B}^{+}\times [-(\frac{3}{4})^2,0],\qquad\bar{v}_{n}=0\quad\text{on}\
\tilde{B}^{+}\times\{s=-(\frac{3}{4})^2\}.
\end{equation*}

\begin{equation*}
\partial_{s}\bar{b}_{n}-\Delta\bar{b}_{n}=-\epsilon_{n}(\bar{v}_{n}\cdot\nabla)\bar{b}_{n}
+\epsilon_{n}(\bar{b}_{n}\cdot\nabla)\bar{v}_{n},\quad
{\rm{div}}\,\bar{v}_{n}=0\qquad \text{in}\ \tilde{Q}^{+},
\end{equation*}
\begin{equation*}
\bar{b}_{n}\cdot \nu=0, \ (\nabla \times \bar{b}_{n})\times
\nu=0\quad\text{on}\
\partial\tilde{B}^{+}\times [-(\frac{3}{4})^2,0],\qquad
\end{equation*}
\[
\bar{b}_{n}=0\quad\text{on}\
\tilde{B}^{+}\times\{s=-(\frac{3}{4})^2\}.
\]
Using the estimate of Stokes system in Lemma \ref{lem1}, we get
\begin{equation}\label{system-10}
\begin{split}
\|\partial_{s}&\bar{v}_{n}\|_{L_{x,t}^{\kappa,\lambda}(\tilde{Q}^{+})}
+\|\partial_{s}\bar{b}_{n}\|_{L_{x,t}^{\kappa,\lambda}(\tilde{Q}^{+})}
+\|\bar{v}_{n}\|_{L^{\kappa}((-(3/4)^2,0);W_{0}^{2,\lambda}(\tilde{B}^{+}))}\\
&+\|\bar{b}_{n}\|_{L^{\kappa}((-(3/4)^2,0);W_{0}^{2,\lambda}(\tilde{B}^{+}))}
+\|\bar{\pi}_{n}\|_{L^{\kappa}((-(3/4)^2,0);W^{1,\lambda}(\tilde{B}^{+}))}\\
\leq &\epsilon_{n}(\|(\bar{v}_{n}\cdot
\nabla)\bar{v}_{n}\|_{L_{x,t}^{\kappa,\lambda}(Q^{+})}
+\|(\bar{b}_{n}\cdot
\nabla)\bar{b}_{n}\|_{L_{x,t}^{\kappa,\lambda}(Q^{+})}\\
&+\|(\bar{v}_{n}\cdot
\nabla)\bar{b}_{n}\|_{L_{x,t}^{\kappa,\lambda}(Q^{+})}
+\|(\bar{b}_{n}\cdot
\nabla)\bar{v}_{n}\|_{L_{x,t}^{\kappa,\lambda}(Q^{+})})\leq
C\epsilon_{n}.
\end{split}
\end{equation}
Next, we define $\tilde{v}_{n}=v_{n}-\bar{v}_{n}$,
$\tilde{b}_{n}=b_{n}-\bar{b}_{n}$ and
$\tilde{\pi}_{n}=\pi_{n}-\bar{\pi}_{n}$. Then it is straightforward
that $\tilde{v}_{n}$, $\tilde{b}_{n}$ and $\tilde{\pi}_{n}$ solve
\begin{equation*}
\begin{cases}
\partial_{s}\tilde{v}_{n}-\Delta\tilde{v}_{n}+\nabla\tilde{\pi}_{n}=0,\quad\text{div}\
\tilde{v}_{n}=0\quad\text{in}\ \tilde{Q}^{+},
\\
\partial_{s}\tilde{b}_{n}-\Delta\tilde{b}_{n}=0,\quad\text{div}\
\tilde{b}_{n}=0\quad\text{in}\ \tilde{Q}^{+},
\end{cases}
\end{equation*}
with boundary data $\tilde{v}_{n}=0$, $\tilde{b}_{n}\cdot \nu=0$ and
$(\nabla \times \tilde{b}_{n})\times \nu=0$ on
$B_1\cap\{x_{3}=0\}\times(-1,0)$.
Using local estimates of Stokes system and heat equation near
boundary, we then note that $\tilde{v}_{n}$, $\tilde{b}_{n}$ and
$\tilde{\pi}_{n}$ satisfy
\[
\|\nabla^{2}\tilde{v}_{n}\|_{L_{x,t}^{\tilde{\kappa},\lambda}(Q^{+}_{9/16})}
+\|\nabla^{2}\tilde{b}_{n}\|_{L_{x,t}^{\tilde{\kappa},\lambda}(Q^{+}_{9/16})}
+\|\nabla\tilde{\pi}_{n}\|_{L_{x,t}^{\tilde{\kappa},\lambda}(Q^{+}_{9/16})}\leq
C(1+\epsilon_{n}),
\]
where $\tilde{\kappa}$ is the number with
$3/\tilde{\kappa}+2/\lambda=1$.

Now, by the Poincar\'{e} inequality, we have
\begin{equation*}
\tilde{P}(\pi_{n},\theta)\leq
C_{2}\Big(\tilde{P}_{1}(\bar{\pi}_{n},\theta)+\tilde{P}_{1}(\tilde{\pi}_{n},\theta)\Big).
\end{equation*}
We note that $P_{1}(\bar{\pi}_{n},\theta)$ goes to zero as
$n\rightarrow\infty$ because of \eqref{system-10}. On the other
hand, using the H\"{o}lder's inequality, we have
\begin{equation*}
P_{1}(\tilde{\pi}_{n},\theta)
\leq\theta^{2}\biggl(\int^{0}_{-\theta^{2}}\Bigl(\int_{B^{+}_{\theta}}
|\nabla\tilde{\pi}|^{\tilde{\kappa}}dy\Bigr)^{\frac{\lambda}{\tilde{\kappa}}}ds\biggr)^{\frac{1}{\lambda}}
\leq C\theta^{2}(1+\epsilon_{n}).
\end{equation*}
Summing up above observations, we obtain
\begin{equation}\label{scaling-3000}
\liminf_{n\rightarrow\infty}
\tilde{P}(\hat{\pi}_{n},\theta)\leq\lim_{n\rightarrow\infty}C_{2}\theta^{2}(1+\epsilon_{n})\leq
C_{2}\theta^{1+\alpha_{0}}.
\end{equation}
Consequently, if a constant $C$ in \eqref{scaling-1000} is taken
bigger than $2(C_{1}+C_{2})$ in \eqref{scaling-2000} and
\eqref{scaling-3000}, this leads to a contradiction, since
\begin{equation*}
2(C_{1}+C_{2})\theta^{1+\alpha_{0}}\leq
C\theta^{1+\alpha_{0}}\leq\liminf_{n\rightarrow\infty}\tau_{n}(\theta)
\leq(C_{1}+C_{2})\theta^{1+\alpha_{0}}.
\end{equation*}
This deduces the Lemma \ref{mhd-decay}.
\end{pf-mhd-decay}

Lemma \ref{mhd-decay} is the crucial part of the proof of
Proposition \ref{ep-regularity} and since its verification is rather
straightforward (compare to \cite[Lemma 7]{GKT06}), the proof of
Proposition \ref{ep-regularity} is omitted.

\subsection{Interior regularity}

In this subsection, we present an interior regularity condition (see
Theorem \ref{main-theorem-interior}) and give its proof. As
mentioned in Introduction, we very recently became to know that the
same result for interior case was obtained in \cite{WZ12}. However,
since the proof of ours is different to that of \cite{WZ12}, we give
our proof.

We first state the main result for interior case.

\begin{theorem}\label{main-theorem-interior}(\textbf{Interior regularity})
Let $(u,b,\pi)$ be a suitable weak solution of the MHD equations
\eqref{MHD} in $\mathbb{R}^3 \times I$. Suppose that for every pair
$p,q$ satisfying $ 1\leq \frac{3}{p}+\frac{2}{q}\leq 2, \ 1\leq
q\leq \infty$, there exists $\epsilon>0$ depending only on $p,q$
such that for some point $z=(x,t)\in \mathbb{R}^3 \times I$ with $u$
is locally in $L^{p,q}_{x,t}$ and
\begin{equation}\label{interior-condition1}
\limsup_{r\rightarrow 0}r^{-(\frac{3}{p}+\frac{2}{q}-1)}
\norm{\norm{u}_{L^p(B_{x,r})}}_{L^q(t-r^2,t)}<\epsilon.
\end{equation}
Then, $u$ and $b$ are regular at $z=(x,t)$
\end{theorem}

\bigskip

Compared to Theorem \ref{main-thm-boundary-v}, we remark that the
range of $q$ in the interior is wider than that of the boundary
case. This is mainly due to difference of estimates of the pressure
for the interior and boundary cases. Since proof of interior case is
simpler than the boundary case, we give the main stream of how the
proof goes.

We first observe that the estimate \eqref{boundary-b} is also valid
for the interior case for $1\leq q\leq \infty$. Since its
verification is rather straightforward, we just state and omit the
details (see also \cite[Lemma 3.7]{KL09}).

\begin{lemma}\label{interior-estimate-b}
Let $z=(x,t)\in\R^3\times I$. Suppose that $u\in
L^{p,q}_{x,t}(Q_{z,r})$ with $\frac{3}{p}+\frac{2}{q}=2$,
$\frac{3}{2} \leq p \leq \infty$. Then for $0<r\leq\rho/4$
\begin{equation}\label{interior-b}
K_{b}(r)\leq C\bigg(\frac{\rho}{r}\bigg)^3G_{u,p,q}^2\Psi(\rho)
+C\bigg(\frac{r}{\rho}\bigg)^2K_{b}(\rho).
\end{equation}
\end{lemma}

Next we estimate the pressure in the interior. We first introduce a
useful invariant functional in the interior case defined as follows:
\[
S(r):=\frac{1}{r^2}\int_{Q_{z,r}}|\pi(y,s)|^{\frac{3}{2}}dyds.
\]
Now we recall an estimate of the pressure involving the functional
above and since its proof is given in \cite[Lemma 3.3]{KL09}, we
just state in the following lemma.

\begin{lemma}\label{interior-pressure}
Let $0<r\leq\rho/4$ and $Q_{\rho} \subset \R^3 \times I$. Then
\begin{equation}\label{estimate-S}
S(r)\leq
C\bigg(\frac{\rho}{r}\biggr)^2\Psi(\rho)+C\bigg(\frac{r}{\rho}\bigg)S(\rho).
\end{equation}
\end{lemma}

\bigskip

The other estimates such as $M_u(r)$ and $\int_{Q_{z,r}}|u| |b|^2
dxdt$ in Lemma \ref{estimate-ub} is also valid in the interior by
following similar arguments, and thus we skip its details. Now we
are ready to give the proof for Theorem \ref{main-theorem-interior}.
Here all invariant functionals in this subsection are defined over
the interior parabolic balls.

\bigskip

\begin{pfthm3}
Under the hypothesis \eqref{interior-condition1}, recalling
Lemma \ref{interior-estimate-b} and Lemma
\ref{estimate-ub}, we note first that for $4r<\rho$
\begin{equation}\label{kb-int-est}
K_{b}(r)\leq C\epsilon^2(\frac{\rho}{r})^3\Psi(\rho)
+C(\frac{r}{\rho})^2K_{b}(\rho)\leq
C\bke{\epsilon^2(\frac{\rho}{r})^3+(\frac{r}{\rho})^2}\Psi(\rho),
\end{equation}
\begin{equation}\label{mu-intmb}
M_u(r)+\frac{1}{r^2}\int_{Q_{z,r}}\abs{u}\abs{b}^2dz\leq
C\epsilon(\frac{\rho}{r})\Psi(\rho).
\end{equation}
Next, due to the pressure estimate \eqref{estimate-S}, we obtain
\[
\frac{1}{r^2}\int_{Q_{z,r}}\abs{u}\abs{\pi} dz \leq
M^{\frac{1}{3}}_u(r)S^{\frac{2}{3}}(r)\leq C\epsilon^{\frac{1}{3}}
(\frac{\rho}{r})^{\frac{5}{3}}\Psi(\rho)+C\epsilon^{\frac{1}{3}}
(\frac{r}{\rho})^{\frac{1}{3}}\Psi^{\frac{1}{3}}(\rho)S^{\frac{2}{3}}(\rho)
\]
\begin{equation}\label{int-pre-est}
\leq C\epsilon^{\frac{1}{3}}
(\frac{\rho}{r})^{\frac{5}{3}}\Psi(\rho)+C\epsilon^{\frac{1}{3}}
(\frac{r}{\rho})^{\frac{1}{3}}S(\rho),
\end{equation}
where Young's inequality is used. Next, combining estimates
\eqref{kb-int-est}-\eqref{int-pre-est}, we have via the local energy
inequality
\[
\Psi(\frac{r}{2})\leq
C\bke{\epsilon^2(\frac{\rho}{r})^3+(\frac{r}{\rho})^2}\Psi(\rho)
+C\epsilon^{\frac{2}{3}}(\frac{\rho}{r})^{\frac{2}{3}}\Psi^{\frac{2}{3}}(\rho)
\]
\begin{equation}\label{est-psi-100}
+C\epsilon(\frac{\rho}{r})\Psi(\rho) +C\epsilon^{\frac{1}{3}}
(\frac{\rho}{r})^{\frac{5}{3}}\Psi(\rho)+C\epsilon^{\frac{1}{3}}
(\frac{r}{\rho})^{\frac{1}{3}}S(\rho)
\end{equation}
where we used Young's inequality and $\frac{1}{r^3}\int_{Q_{z,r}}
\abs{u}^2\leq CM^{\frac{2}{3}}_u(r)$. Let $\epsilon_3$ be a small
positive number, which will be specified later. Now via estimates
\eqref{estimate-S} and \eqref{est-psi-100} we consider
\[
\Psi(\frac{r}{2})+\epsilon_3 S(\frac{r}{2}) \leq \epsilon+
C\bke{\epsilon^2(\frac{\rho}{r})^3+(\frac{r}{\rho})^2
+\epsilon_3(\frac{\rho}{r})^2}\Psi(\rho)
+C\bke{\epsilon^{\frac{1}{3}}
(\frac{r}{\rho})^{\frac{1}{3}}+\epsilon_3(\frac{r}{\rho})}S(\rho),
\]
where we used Young's inequality. We fix $\theta \in
(0,\frac{1}{4})$ with $\theta<\frac{1}{8^3(C+1)^3}$ and then
$\epsilon_3$ and $\epsilon$ are taken to satisfy
\[
0<\epsilon_3<\min\bket{\frac{\theta^2}{16C^2}},\qquad
0<\epsilon<\min\bket{\frac{\epsilon^*}{4}, \,\,\epsilon^3_3,\,\,
\frac{\theta^{\frac{3}{2}}}{4C^{\frac{1}{2}}}},
\]
where $\epsilon^*$ is the number introduced in \cite[Theorem
1.1]{KL09}. We then obtain
\[
\Psi(\theta r)+\epsilon_3 S(\theta r)\leq
\frac{\epsilon^*}{4}+\frac{1}{4}\Big(\Psi(r)+\epsilon_3 S(r)\Big).
\]
Usual method of iteration implies that there exists a sufficiently
small $r_0>0$ such that for all $r<r_0$
\[
\Psi(r)+\epsilon_3 S(r)\leq \frac{\epsilon^*}{2}.
\]
This completes the proof.
\end{pfthm3}

\section*{Acknowledgments}
K. Kang's work was partially supported by NRF-2012R1A1A2001373.
J.-M. Kim's work was partially supported by KRF-2008-331-C00024 and
NRF-2009-0088692.
 The authors wish to
expresses our appreciation to Professor Tai-Peng Tsai for useful
comments.

\begin{equation*}
\left.
\begin{array}{cc}
{\mbox{Kyungkeun Kang}}\qquad&\qquad {\mbox{Jae-Myoung Kim}}\\
{\mbox{Department of Mathematics }}\qquad&\qquad
 {\mbox{Department of Mathematics}} \\
{\mbox{Yonsei University
}}\qquad&\qquad{\mbox{Sungkyunkwan University}}\\
{\mbox{Seoul, Republic of Korea}}\qquad&\qquad{\mbox{Suwon, Republic of Korea}}\\
{\mbox{kkang@yonsei.ac.kr }}\qquad&\qquad {\mbox{cauchy@skku.edu }}
\end{array}\right.
\end{equation*}

\end{document}